\documentclass[12pt,draft]{article}
\usepackage{amsmath,amssymb,amsthm}
\usepackage{color}
\usepackage{bm} 
\usepackage{ascmac} 

\newtheorem{thm}{Theorem}[section]
\newtheorem{crl}[thm]{Corollary}

\newtheorem{prp}[thm]{Proposition}
\newtheorem{lmm}[thm]{Lemma}
\newtheorem{rmk}[thm]{Remark}
\newtheorem{dfn}[thm]{Definition}

\newtheorem{asm}[thm]{Assumption}

\newcommand{\LR}{\Leftrightarrow}

\newcommand{\ra}{\rightarrow}
\newcommand{\mt}{\mapsto}

\newcommand{\HKC}{\mathrm{Hom}(K,\mathbb C)}

\newcommand{\oq}{\overline{\mathbb Q}}

\newcommand{\bq}{B_{\mathrm cris}\overline{\mathbb Q_p}}
\newcommand{\wg}{W_p}

\newcommand{\op}{\mathrm{ord}_p\hspace{1pt}}
\newcommand{\tst}{\text{ s.t.\ }}

\makeatletter
\newcommand{\subjclass}[2][2010]{%
  \let\@oldtitle\@title%
  \gdef\@title{\@oldtitle\footnotetext{#1 \emph{Mathematics subject classification(s).} #2}}%
}
\newcommand{\keywords}[1]{%
  \let\@@oldtitle\@title%
  \gdef\@title{\@@oldtitle\footnotetext{\emph{Key words and phrases.} #1.}}%
}
\makeatother

\title{Note on Coleman's formula for the absolute Frobenius on Fermat curves}
\author{Tomokazu Kashio\thanks{Tokyo University of Science, \texttt{kashio\_tomokazu@ma.noda.tus.ac.jp}}}

\begin{document}


\maketitle

\begin{abstract}
Coleman calculated the absolute Frobenius on Fermat curves explicitly.
In this paper we show that a kind of $p$-adic continuity implies a large part of his formula.
To do this, we study a relation between functional equations of the ($p$-adic) gamma function and monomial relations on ($p$-adic) CM-periods.
\end{abstract}

\section{Introduction}

We modify Euler's gamma function $\Gamma(z)$ into 
\begin{align*}
\Gamma_\infty(z):=\frac{\Gamma(z)}{\sqrt{2\pi}}=\exp(\zeta'(0,z)) \quad (z>0)
\end{align*}
and focus on its special values at rational numbers.
Here we put $\zeta(s,z):=\sum_{k=0}^\infty (z+k)^{-s}$ to be the Hurwitz zeta function.
The last equation is due to Lerch. One has a``simple proof'' in \cite[p17]{Yo}.
The gamma function enjoys some functional equations:
\begin{align}
&\text{Euler's Reflection formula:}&&\Gamma_\infty(z)\Gamma_\infty(1-z)=\frac{1}{2\sin \pi z}, \label{erf} \\
&\text{Gauss' Multiplication formula:}&&\prod_{k=0}^{d-1}\Gamma_\infty(z+\tfrac{k}{d})=d^{\frac{1}{2}-dz}\Gamma_\infty(dz) \quad (d \in \mathbb N). \label{gmf}
\end{align}
For proofs, see \cite[\S 3, 4]{Ar}.
The main topic of this paper is a relation between such functional equations and monomial relations of CM-periods, and its $p$-adic analogue.
We introduce some notations.

\begin{dfn}
Let $K$ be a CM-field.
We denote by $I_K$ the $\mathbb Q$-vector space formally generated by all complex embeddings of $K$:
\begin{align*}
I_K:=\bigoplus_{\sigma \in \HKC} \mathbb Q \cdot \sigma.
\end{align*}
We identify a subset $S\subset \HKC$ as an element $\sum_{\sigma \in S}\sigma\in I_K$.
Shimura's period symbol is the bilinear map
\begin{align*}
p_K\colon I_K \times I_K \ra \mathbb C^\times/\oq^\times
\end{align*}
characterized by the following properties $(\mathrm{P}_1)$, $(\mathrm{P}_2)$.
\begin{enumerate}
\item[$(\mathrm{P}_1)$] Let $A$ be an abelian variety defined over $\oq$, having CM of type $(K,\Xi)$.
Namely, for each $\sigma \in \HKC$, there exists a non-zero ``$K$-eigen'' differential form $\omega_\sigma$ of the second kind satisfying 
\begin{align*}
k^*(\omega_\sigma)=\sigma(k)\omega_\sigma \quad (k\in K), 
\end{align*}
where $k^*$ denotes the action of $k \in K$ via $K\cong \mathrm{End}(A)\otimes_\mathbb Z \mathbb Q$ on the de Rham cohomology $H^1_{\mathrm dR}(A,\mathbb C)$.
Then we have
\begin{align*}
&\Xi=\{\sigma \in \HKC \mid \omega_\sigma \text{ is holomorphic}\}, \\
&p_K(\sigma,\Xi)\equiv
\begin{cases}
\pi^{-1}\int_\gamma \omega_\sigma  & (\sigma \in \Xi) \\
\int_\gamma \omega_\sigma & (\sigma \in \HKC-\Xi)
\end{cases}
\mod \oq^\times
\end{align*}
for an arbitrary closed path $\gamma \subset A(\mathbb C)$ satisfying $\int_\gamma \omega_\sigma  \neq 0$.
\item[$(\mathrm{P}_2)$] Let $\rho$ be the complex conjugation. Then we have
\begin{align*}
p_K(\sigma,\tau)p_K(\rho \circ \sigma, \tau)\equiv p_K(\sigma,\tau)p_K(\sigma,\rho \circ \tau)\equiv 1 \mod \oq^\times \quad (\sigma,\tau \in \HKC).
\end{align*}
\end{enumerate}
\end{dfn}

Strictly speaking, Shimura's $p_K$ in \cite[\S 32]{Sh} is a bilinear map on $\bigoplus_{\sigma \in \HKC} \mathbb Z \cdot \sigma$.
The period symbol also enjoys the following relations:
\begin{enumerate}
\item[$(\mathrm{P}_3)$] Let $\iota\colon K'\cong K$ be an isomorphism of CM-fields. Then we have
\begin{align*}
p_K(\sigma,\tau) \equiv p_{K'}( \sigma \circ \iota,\tau \circ \iota) \mod \oq^\times  \quad (\sigma,\tau \in \HKC).
\end{align*}
\item[$(\mathrm{P}_4)$] Let $K \subset L$ be a field extension of CM-fields. 
We define two linear maps defined as 
\begin{align*}
&\mathrm{Res}\colon I_L \ra I_K,\ \tilde \sigma \mt \tilde \sigma|_K \quad (\tilde \sigma \in \mathrm{Hom}(L,\mathbb C)), \\
&\mathrm{Inf}\colon I_K \ra I_L,\ \sigma \mt \sum_{\substack{\tilde \sigma \in \mathrm{Hom}(L,\mathbb C) \\ \tilde\sigma|_K=\sigma}}
\tilde \sigma \quad (\sigma \in \HKC).
\end{align*}
Then we have
\begin{align*}
p_K(\mathrm{Res}(X),Y)\equiv 
p_L(X,\mathrm{Inf}(Y)) \mod \oq^\times \quad (X \in I_L,\ Y \in I_K).
\end{align*}
\end{enumerate}

The following results by Gross-Rohrlich and the above relations $(\mathrm{P}_3)$, $(\mathrm{P}_4)$ provide 
an explicit formula \cite[Theorem 2.5, Chap.~III]{Yo} on $p_K$ for $K=\mathbb Q(\zeta_N)$ ($\zeta_N=e^{\frac{2\pi i}{N}}$, $N\geq 3$).
We can rewrite it in the form (\ref{ef}) by the arguments in \cite[\S 6]{Ka2}.
Let $\sigma_ b\in \mathrm{Gal}(\mathbb Q(\zeta_N)/\mathbb Q)$ ($(b,N)=1$) be defined by $\sigma_b(\zeta_N):=\zeta_N^b$,
$\langle \alpha \rangle \in (0,1)$ denote the fraction part of $\alpha \in \mathbb Q-\mathbb Z$.

\begin{thm}[{\cite[Theorem in Appendix]{Gr}}]
Let $F_N:x^N+y^N=1$ be the $N$th Fermat curve, $\eta_{r,s}:=x^{r-1}y^{s-N}dx$ its differential forms of the second kind $(0<r,s<N$, $r+s\neq N)$.
Then we have for any closed path $\gamma$ on $F_N(\mathbb C)$ with $\int_\gamma \eta_{r,s}\neq 0$
\begin{align} \label{gamma}
\int_\gamma \eta_{r,s}\equiv \frac{\Gamma(\frac{r}{N})\Gamma(\frac{s}{N})}{\Gamma(\frac{r+s}{N})} \mod \mathbb Q(\zeta_N)^\times.
\end{align}
\end{thm}

\begin{thm}[{\cite[\S 2]{Gr}, \cite[\S 2, Chap.~III]{Yo}}] \label{ptoint}
The CM-type corresponding to $\eta_{r,s}$ is 
\begin{align} \label{xi}
\Xi_{r,s}:=\{\sigma_b \mid 1\leq b \leq N,\ (b,N)=1,\ \langle \tfrac{br}{N}\rangle+\langle \tfrac{bs}{N}\rangle+\langle \tfrac{b(N-r-s)}{N}\rangle=1\}.
\end{align}
That is, we have 
\begin{align*}
p_{\mathbb Q(\zeta_N)}(\mathrm{id},\Xi_{r,s}) \equiv 
\begin{cases}
\pi^{-1} \int_{\gamma} \eta_{r,s}  & (r+s<N) \\
\int_{\gamma} \eta_{r,s} &(r+s>N)
\end{cases}
\mod \oq^\times.
\end{align*}
\end{thm}

\begin{crl}[{\cite[Theorem 3]{Ka2}}]
We have for any $\frac{a}{N} \in \mathbb Q -\mathbb Z$
\begin{align} \label{ef}
\Gamma_\infty(\tfrac{a}{N}) \equiv 
\pi^{\frac{1}{2}-\langle\frac{a}{N}\rangle} p_{\mathbb Q(\zeta_N)}\left(\mathrm{id},\sum_{(b,N)=1} \left(\tfrac{1}{2}-\langle \tfrac{ab}{N}\rangle\right) \cdot \sigma_b\right) 
\mod \oq^\times.
\end{align}
Here the sum runs over all $b$ satisfying $1\leq b \leq N$, $(b,N)=1$.
\end{crl}

Note that (\ref{ef}) holds true even if $(a,N)>1$, essentially due to $(\mathrm{P}_4)$.
Although the following is just a toy problem, we provide its proof by using the period symbol, in order to explain the theme of this paper:
we may say that some functional equations of the gamma function ``correspond'' to some monomial relations of CM-periods.

\begin{prp}[A toy problem] \label{toy}
The explicit formula {\rm (\ref{ef})} implies the following ``functional equations $\bmod \oq^\times$'' on $\Gamma(\frac{a}{N})$:
\begin{align*}
&\text{\rm ``Reflection formula'':}&&\Gamma_\infty(\tfrac{a}{N})\Gamma_\infty(\tfrac{N-a}{N}) \equiv 1 \mod \oq^\times, \\
&\text{\rm ``Multiplication formula'':}&&\prod_{k=0}^{d-1}\Gamma_\infty(\tfrac{a}{N}+\tfrac{k}{d}) \equiv \Gamma_\infty(\tfrac{da}{N}) \mod \oq^\times.
\end{align*}
\end{prp}

\begin{proof}
``Reflection formula'' follows from $(\mathrm{P}_2)$ immediately.
Concerning ``Multiplication formula'', we may assume that $d\mid N$.
Under the expression {\rm (\ref{ef})}, ``Multiplication formula'' is equivalent to
\begin{multline*}
\pi^{\overset{d-1}{\underset{k=0}{\sum}} \frac{1}{2}-\langle \frac{a}{N}+\frac{k}{d}\rangle}
p_{\mathbb Q(\zeta_N)}\left(\mathrm{id}, \sum_{(b,N)=1}\left(\sum_{k=0}^{d-1} \tfrac{1}{2}-\langle \tfrac{ab}{N}+\tfrac{kb}{d}\rangle\right) \cdot \sigma_b\right) \\
\equiv \pi^{\frac{1}{2}-\langle \frac{ad}{N}\rangle}
p_{\mathbb Q(\zeta_N)}\left(\mathrm{id},\sum_{(b,N)=1} (\tfrac{1}{2}-\langle \tfrac{dab}{N}\rangle) \cdot \sigma_b \right).
\end{multline*}
This follows from the multiplication formula 
\begin{align} \label{mffb1}
\sum_{k=0}^{d-1}B_1(x+\tfrac{k}{d})=B_1(dx)
\end{align}
for the $1$st Bernoulli polynomial $B_1(x)=x-\frac{1}{2}$.
\end{proof}

The aim of this paper is to study a $p$-adic analogue of such ``correspondence''. 
More precisely, we shall characterize the $p$-adic gamma function by its functional equations and some special values.
Then we show that the period symbol and its $p$-adic analogue satisfy the corresponding properties to such functional equations.
As an application, we provide an alternative proof of a large part of Coleman's formula (Theorem \ref{Cf}-(i)):
originally, Coleman's formula was proved by calculating the absolute Frobenius on {\em all Fermat curves}.
We shall see that it suffices to calculate it on {\em only one curve} (Remark \ref{onecurve}).

\begin{rmk}
Yoshida and the author formulated conjectures in {\rm \cite{KY1,KY2,Ka2}} which are generalizations of Coleman's formula, from cyclotomic fields to arbitrary CM-fields:
Coleman's formula implies ``the reciprocity law on cyclotomic units'' {\rm \cite{Ka1}} 
and ``the Gross-Koblitz formula on Gauss sums'' {\rm \cite{GK,Co1}} simultaneously.
The author conjectured a generalization {\rm \cite[Conjecture 4]{Ka2}} of Coleman's formula which implies a part of Stark's conjecture and a generalization of 
(the rank $1$ abelian) Gross-Stark conjecture simultaneously. 
The results in this paper (in particular {\rm Remark \ref{onecurve}}) are very important toward this generalization,
since we know only a finite number of algebraic curves (e.g., {\rm \cite{BS}}) whose Jacobian varieties have CM by CM-fields which are not abelian over $\mathbb Q$.
\end{rmk}

The outline of this paper is as follows.
First we introduce Coleman's formula \cite{Co2} for  the absolute Frobenius on Fermat curves in \S \ref{secCf}.
The author rewrote it in the form of Theorem \ref{Cf}: roughly speaking, we write Morita's $p$-adic gamma function $\Gamma_p$ 
in terms of Shimura's period symbol $p_K$, its $p$-adic analogue $p_{K,p}$, and modified Euler's gamma function $\Gamma_\infty$.
In \S \ref{secmr}, we show that some functional equations almost characterize $\Gamma_p$ (Corollary \ref{crloffe}), 
and the corresponding properties ((\ref{mf4G}), Theorem \ref{mthm}) hold for $p_K,p_{K,p},\Gamma_\infty$.
Then we see that a large part (Corollary \ref{maincrl}) of Coleman's formula follows automatically, without explicit computation, under assuming certain $p$-adic continuity properties.
Unfortunately, our results have a root of unity ambiguity although the original formula is a complete equation, since some definitions are well-defined only up to roots of unity.
In \S \ref{secpc}, we confirm that we can show (at least, a part of) needed $p$-adic continuity properties relatively easily.

\section{Coleman's formula in terms of period symbols} \label{secCf}

Coleman explicitly calculated the absolute Frobenius on Fermat curves \cite{Co2}.
The author rewrote his formula in \cite{Ka1, Ka2} as follows.

\subsection{$p$-adic period symbol}

Let $p$ be a rational prime, $\mathbb C_p$ the $p$-adic completion of the algebraic closure $\overline{\mathbb Q_p}$ of $\mathbb Q_p$, and $\mu_\infty$ the group of all roots of unity.
For simplicity, we fix embeddings $\oq \hookrightarrow \mathbb C,\mathbb C_p$
and consider any number field as a subfield of each of them.
Let $B_{\mathrm cris}\subset B_{\mathrm dR}$ be Fontaine's $p$-adic period rings.
We consider the composite ring $\bq \subset B_{\mathrm dR}$.
Let $A$ be an abelian variety with CM defined over $\oq$, $\gamma$ a closed path on $\subset A(\mathbb C)$, and $\omega$ a differential form of the second kind of $A$. 
Then the $p$-adic period integral 
\begin{align*}
\int_p\colon H^{\mathrm B}_1(A(\mathbb C),\mathbb Q) \times H_{\mathrm dR}^1(A,\oq) \ra \bq,\ (\gamma,\omega)\mt \int_{\gamma,p}\omega
\end{align*}
is defined by the comparison isomorphisms of $p$-adic Hodge theory, instead of the de Rham isomorphism (e.g., \cite[\S 6]{Ka1}, \cite[\S 5.1]{Ka2}).
Here $H^B$ denotes the singular (Betti) homology.
Then, in a similar manner to $p_K$, we can define the $p$-adic period symbol
\begin{align*}
p_{K,p}\colon I_K \times I_K \ra (\bq-\{0\})^\mathbb Q/\oq^\times
\end{align*}
satisfying $p$-adic analogues of $(\mathrm{P}_1)$, $(\mathrm{P}_2)$, $(\mathrm{P}_3)$, $(\mathrm{P}_4)$.
Here we put $(\bq-\{0\})^\mathbb Q:=\{x \in B_{\mathrm dR} \mid \exists n \in \mathbb N \tst x^n \in \bq-\{0\}\}$.
Moreover the ``ratio'' 
\begin{align*}
\left[\int_\gamma \omega_\sigma:\int_{\gamma,p} \omega_\sigma\right] \in (\mathbb C^\times \times (\bq-\{0\}))/\oq^\times
\end{align*}
depends only on $\sigma \in \HKC$ and the CM-type $\Xi$. That is, if we replace $A,\omega_\sigma,\gamma$ with $A',\omega_\sigma',\gamma'$ for the same $\Xi,\sigma$, then we have
\begin{align*}
\frac{\int_{\gamma'} \omega_\sigma'}{\int_\gamma \omega_\sigma}=
\frac{\int_{\gamma',p} \omega_\sigma'}{\int_{\gamma,p} \omega_\sigma} \in \oq^\times.
\end{align*}
Therefore we may consider the following ratio of  the symbols $[p_K:p_{K,p}]$, which is well-defined up to $\mu_\infty$.

\begin{prp}[{\cite[Proposition 4]{Ka2}}] \label{tips}
There exists a bilinear map 
\begin{align*}  
[p_K:p_{K,p}]\colon I_K \times I_K \ra (\mathbb C^\times \times(\bq-\{0\})^\mathbb Q)/(\mu_\infty\times \mu_\infty)\oq^\times
\end{align*}
satisfying the following.
\begin{enumerate}
\item Let $A,\Xi,\sigma,\omega_\sigma,\gamma$ be as in $(\mathrm{P}_1)$. Then
\begin{multline*}
[p_K:p_{K,p}](\sigma,\Xi) \\
\equiv
\begin{cases}
[(2\pi i)^{-1}\int_\gamma \omega_\sigma:(2\pi i)_p^{-1}\int_{\gamma,p} \omega_\sigma] & (\sigma \in \Xi) \\
[\int_\gamma \omega_\sigma:\int_{\gamma,p} \omega_\sigma] & (\sigma \in \HKC-\Xi)
\end{cases}
\mod (\mu_\infty\times \mu_\infty)\oq^\times.
\end{multline*}
Here $(2\pi i)_p \in B_{\mathrm cris}$ is the $p$-adic counterpart of $2\pi i$ defined in, e.g., {\rm \cite[\S 5.1]{Ka2}}.
\item We have for $\sigma,\tau \in \HKC$ and for the complex conjugation $\rho$
\begin{align*}
&[p_K:p_{K,p}](\sigma,\tau) \cdot [p_K:p_{K,p}] (\rho \circ \sigma,\tau)\equiv 1 \mod (\mu_\infty\times \mu_\infty)\oq^\times, \\
&[p_K:p_{K,p}](\sigma,\tau) \cdot [p_K:p_{K,p}] (\sigma,\rho \circ \tau)\equiv 1 \mod (\mu_\infty\times \mu_\infty)\oq^\times.
\end{align*}
\item Let $\iota\colon K'\cong K$ be an isomorphism of CM-fields. Then we have for $\sigma,\tau \in \HKC$
\begin{align*}
[p_K:p_{K,p}](\sigma,\tau) \equiv[p_{K'}:p_{K',p}]( \sigma \circ \iota,\tau \circ \iota) \mod (\mu_\infty\times \mu_\infty)\oq^\times.
\end{align*}
\item Let $K \subset L$ be a field extension of CM-fields. 
Then we have for $X \in I_L$, $Y \in I_K$
\begin{align*}
[p_K:p_{K,p}](\mathrm{Res}(X),Y)\equiv 
[p_L:p_{L,p}](X,\mathrm{Inf}(Y)) \mod (\mu_\infty\times \mu_\infty)\oq^\times.
\end{align*}
\end{enumerate}
\end{prp}

\subsection{Coleman's formula}

Theorem \ref{Cf} below is essentially due to Coleman \cite[Theorems 1.7, 3.13]{Co2}.
Note that the original formula does not have a root of unity ambiguity.
First we prepare some notations.
We assume that $p$ is an odd prime.

\begin{dfn} \label{dfn}
\begin{enumerate}
\item Let $\mathbb C_p^1:=\{z \in \mathbb C_p^\times \mid |z|_p=1\}$.
We fix a group homomorphism 
\begin{align*}
\exp_p \colon \mathbb C_p \ra \mathbb C_p^1
\end{align*}
which coincides with the usual power series $\exp_p(z):=\sum_{k=0}^\infty \frac{z^k}{k!}$ on the convergence region. 
For $\alpha \in \mathbb C_p^\times$, $\beta \in \mathbb C_p$, we put
\begin{align*}
\alpha^\beta:=\exp_p(\beta \log_p \alpha)
\end{align*}
with $\log_p$ Iwasawa's $p$-adic $\log$ function.
\item For $z \in \mathbb C_p^\times$, we put
\begin{align*}
z^*:= \exp_p(\log_p(z)), && z^\flat:=p^{\op z}z^*.
\end{align*}
Here we define $\op z \in \mathbb Q$ by $|z|_p=|p|_p^{\op z}$.
Note that $z\equiv z^\flat \bmod \mu_\infty$ $(z\in \mathbb C_p^\times)$.
\item We define the $p$-adic gamma function on $\mathbb Q_p$ as follows.
\begin{enumerate}
\item On $\mathbb Z_p$, $\Gamma_p(z)$ denotes Morita's $p$-adic gamma function 
which is the unique continuous function $\Gamma_p\colon \mathbb Z_p \ra \mathbb Z_p^\times$ satisfying
\begin{align*}
\Gamma_p(n):=(-1)^n \prod_{1\leq k \leq n-1,\ p\nmid k} k \quad(n\in \mathbb N).
\end{align*}
\item On $\mathbb Q_p-\mathbb Z_p$, we use $\Gamma_p\colon\mathbb Q_p-\mathbb Z_p \ra \mathcal O_{\overline{\mathbb Q_p}}^\times$ defined in 
{\rm\cite[Lemma 4.2]{Ka1}},
which is a continuous function satisfying
\begin{align*}
\Gamma_p(z+1)=z^*\Gamma_p(z),\ \Gamma_p(2z)=2^{2z-\frac{1}{2}}\Gamma_p(z)\Gamma_p(z+\tfrac{1}{2}).
\end{align*}
Such a continuous function on $\mathbb Q_p-\mathbb Z_p$ is unique up to multiplication by $\mu_\infty$.
\end{enumerate}
\item For $z\in \mathbb Z_p$, we define $z_0 \in \{1,2,\dots,p\}$, $z_1\in \mathbb Z_p$ by
\begin{align*}
z=z_0+pz_1.
\end{align*}
Note that when $p\mid z$, we put $z_0=p$, instead of $0$. 
\item Let $\wg$ be the Weil group defined as
\begin{align*}
\wg:=\{\tau \in \mathrm{Gal}(\overline{\mathbb Q_p}/\mathbb Q_p) \mid \tau|_{\mathbb Q_p^{ur}}=\sigma_p^{\deg \tau} \text{ with } \deg \tau \in \mathbb Z\}.
\end{align*}
Here $\mathbb Q_p^{ur}$ denotes the maximal unramified extension of $\mathbb Q_p$,
$\sigma_p$ the Frobenius automorphism on $\mathbb Q_p^{ur}$.
\item We define the action of $\wg$ on $\mathbb Q\cap [0,1)$ by identifying $\mathbb Q\cap [0,1)= \mu_\infty$. Namely
\begin{align*}
\tau(\tfrac{a}{N}):=\tfrac{b}{N} \quad \text{if} \quad \tau(\zeta_N^a)=\zeta_N^b \quad(\tau \in \wg).
\end{align*}
\item Let $\Phi_{\mathrm cris}$ be the absolute Frobenius automorphism on $B_{\mathrm cris}$.
We consider the following action of $\wg$ on $\bq\cong B_{\mathrm cris}\otimes_{\mathbb Q^{ur}} \overline{\mathbb Q_p}$:
\begin{align*}
\Phi_\tau:=\Phi_{\mathrm cris}^{\deg \tau}\otimes \tau \quad(\tau \in \wg).
\end{align*}
\item For $\frac{a}{N}\in \mathbb Q \cap (0,1)$ we put
\begin{multline*}
P(\tfrac{a}{N}) 
:=\frac{\Gamma_\infty(\frac{a}{N}) 
\cdot (2\pi i)_p^{\frac{1}{2}-\langle\frac{a}{N}\rangle} 
p_{\mathbb Q(\zeta_N),p}\left(\mathrm{id},\sum_{(b,N)=1} \left(\frac{1}{2}-\langle \frac{ab}{N}\rangle\right) \sigma_b\right)}
{(2\pi i)^{\frac{1}{2}-\langle\frac{a}{N}\rangle} p_{\mathbb Q(\zeta_N)}\left(\mathrm{id},\sum_{(b,N)=1} \left(\frac{1}{2}-\langle \frac{ab}{N}\rangle\right) \sigma_b\right)} \\
\in (\bq-\{0\})^\mathbb Q/\mu_\infty.
\end{multline*}
This definition makes sense since 
\begin{align*}
\frac{\Gamma_\infty(\frac{a}{N})}
{(2\pi i)^{\frac{1}{2}-\langle\frac{a}{N}\rangle} p_{\mathbb Q(\zeta_N)}\left(\mathrm{id},\sum_{(b,N)=1} \left(\frac{1}{2}-\langle \frac{ab}{N}\rangle\right) \sigma_b\right)} 
\in \oq \subset \bq
\end{align*}
by {\rm(\ref{ef})} and the ratio $[p_K:p_{K,p}]$ is well-defined up to $\mu_\infty$ by {\rm Proposition \ref{tips}}.
\end{enumerate}
\end{dfn}

\begin{rmk} \label{rmk4dfn}
\begin{enumerate}
\item Let $\mu_{p-1}$ be the group of all $(p-1)$st roots of unity, $p^\mathbb Z:=\{p^n \mid n \in \mathbb Z\}$, $1+p\mathbb Z_p:=\{1+pz \mid z \in \mathbb Z_p\}$.
Then we have the canonical decomposition 
\begin{align*}
\begin{array}{ccccccc}
\mathbb Q_p^\times &\ra & \mu_{p-1} &\times& p^\mathbb Z &\times& 1+ p\mathbb Z_p, \\
z &\mt& (\omega(zp^{-\op z})&,&p^{\op z}&,&z^*),
\end{array}
\end{align*}
where $\omega$ denotes the Teichm\"uller character.
The maps $z \mt z^*,z^\flat$ provide a similar (but non-canonical) decomposition of $\mathbb C_p^\times$.
Moreover, we note that the maps $z \mt \exp_p(z),z^*,z^\flat$ are continuous homomorphisms.
\item We easily see that
\begin{align*}
\tau(z)=\langle p z\rangle,\ \tau^{-1}(z)=z_1+1 \quad (z\in \mathbb Z_{(p)}\cap(0,1),\ \tau \in \wg,\ \deg \tau=1).
\end{align*}
\end{enumerate}
\end{rmk}

\begin{thm}[{\cite[Theorem 3]{Ka2}}] \label{Cf}
Let $p$ be an odd prime.
\begin{enumerate}
\item Assume that $z\in \mathbb Z_{(p)}\cap(0,1)$. Then we have
\begin{align*}
\Gamma_p(z)\equiv p^{\frac{1}{2}-\tau^{-1}(z)}\frac{P(z)}{\Phi_{\tau}(P(\tau^{-1}(z)))} \bmod \mu_\infty
\quad (\tau\in \wg,\ \deg \tau=1).
\end{align*}
\item Assume that $z\in (\mathbb Q-\mathbb Z_{(p)})\cap(0,1)$. Then we have
\begin{align*}
\frac{\Gamma_p(\tau(z))}{\Gamma_p(z)}\equiv \frac{p^{(z-\tau(z))\op z}P(\tau(z))}{\Phi_\tau (P(z))}\mod \mu_\infty \quad (\tau\in \wg).
\end{align*}
\end{enumerate}
\end{thm}

\begin{rmk}
As a result, we see that the right-hand sides of {\rm Theorem \ref{Cf}-(i), (ii)} are $p$-adic continuous on $z$, $(z,\tau(z))$ respectively, since the left-hand sides are so.
We use only the $p$-adic continuity in the next section, in order to recover {\rm Theorem \ref{Cf}-(i)}.
\end{rmk}

\section{Main results} \label{secmr}

Morita's $p$-adic gamma function $\Gamma_p \colon \mathbb Z_p \ra \mathbb Z_p^\times$ is the unique continuous function satisfying
\begin{align} \label{cofgp}
\Gamma_p(0)=1,\ 
\frac{\Gamma_p(z+1)}{\Gamma_p(z)}=
\begin{cases}
-z & (z \in \mathbb Z_p^\times), \\
-1 & (z \in p\mathbb Z_p).
\end{cases}
\end{align}  
In this section, we study other functional equations characterizing $\Gamma_p$ and 
provide an alternative proof of Coleman's formula in the case $z \in \mathbb Z_{(p)}$.
Strictly speaking, 
we only ``assume'' that the right-hand sides of Theorem \ref{Cf}-(i), (ii) are continuous on $z$, $(z,\tau(z))$ respectively (of course, this is correct).
Then we can recover a ``large part'' (Corollary \ref{maincrl}) of Theorem \ref{Cf}-(i).
We assume that $p$ is an odd prime.

\subsection{A characterization of Morita's $p$-adic gamma function}

$\Gamma_p(z)$ satisfies the following $p$-adic analogues of multiplication formulas, which we consider only up to roots of unity in this paper.
For the detailed formulation and its proof, see \cite[``Basic properties of $\Gamma_p$'' in \S 2 of Chap.~IV]{Ko}.

\begin{prp} \label{pmult}
Let $d \in \mathbb N$ with $p\nmid d$. Then we have for $z \in \mathbb Z_p$
\begin{align} \label{dmult}
\prod_{k=0}^{d-1} \Gamma_p(z+\tfrac{k}{d}) \equiv d^{1-dz+(dz)_1} \Gamma_p(dz) \mod \mu_\infty.
\end{align}
\end{prp}  

Note that if $p \mid d$, then $z +\frac{k}{d}$ is not in the domain of definition of Morita's $\Gamma_p$.
In the rest of this subsection, we show that multiplication formulas (\ref{dmult})
and some conditions characterize Morita's $p$-adic gamma function (at least up to $\mu_\infty$).

\begin{prp} \label{feef}
Assume a continuous function $f(z) \colon \mathbb Z_p \ra \mathbb C_p^\times$ satisfies 
\begin{align} \label{fe}
\prod_{k=0}^{d-1} f(z+\tfrac{k}{d})\equiv f(dz) \mod \mu_\infty \quad (p\nmid d).
\end{align}
Then the following holds.
\begin{enumerate}
\item $\frac{f(z+1)}{f(z)} \bmod \mu_\infty$ depends only on $\op z$.
\item The values  
\begin{align*}
c_k:=\left(\frac{f(p^k+1)}{f(p^k)}\right)^\flat
\end{align*}
characterize the function $f(z)$ up to $\mu_\infty$. More precisely, for $z \in \mathbb Z_p$, we write the $p$-adic expansion of $z-1$ as
\begin{align*}
z-1=\sum_{k=0}^\infty x_k p^k \quad (x_k \in \{0,1,\dots,p-1\}). 
\end{align*}
Then we have
\begin{align*}
f(z) \equiv \prod_{k=0}^\infty \alpha_k^{x_k-\frac{p-1}{2}} \mod \mu_\infty \quad \text{with} \quad \alpha_k:=c_k \prod_{i=0}^{k-1}c_i^{p^{k-1-i}(p-1)}.
\end{align*}
\end{enumerate}
Conversely, assume that
\begin{align} \label{explicit}
f\left(1+\sum_{k=0}^\infty x_k p^k\right) \equiv \prod_{k=0}^\infty \alpha_k^{x_k-\frac{p-1}{2}} \mod \mu_\infty \quad (x_k \in \{0,1,\dots,p-1\}) 
\end{align}
for constants $\alpha_k \in \mathbb C_p^\times$ satisfying $\alpha_k \ra 1$ $(k \ra \infty)$.
Then $f(z)$ satisfies the functional equations {\rm (\ref{fe})}.
\end{prp} 

\begin{proof}
We suppress $\bmod \mu_\infty$. 
Assume (\ref{fe}).
Replacing $z$ with $z+\frac{1}{d}$, we obtain $\prod_{k=1}^{d} f(z+\frac{k}{d}) \equiv f(dz+1)$. It follows that $\frac{f(z+1)}{f(z)}\equiv \frac{f(dz+1)}{f(dz)}$.
That is,
\begin{align*}
g(z):=\frac{f(z+1)}{f(z)}\equiv g(dz)\quad (p\nmid d \in \mathbb N).
\end{align*}
Then the assertion (i) is clear.
Let $c_k:=(g(p^k))^\flat$, $a_n:=x_0+x_1p+\dots+x_np^n$ ($0\leq x_i \leq p-1$). 
We easily see that
\begin{align*}
\# \{y =1,2,\dots,a_n \mid \op y=k\}=x_k+\sum_{i=k+1}^n x_i p^{i-k-1}(p-1) \quad (0\leq k\leq n).
\end{align*}
Then we can write
\begin{align*}
f(a_n+1)^\flat=(f(1)g(1)g(2)\dotsm g(a_n))^\flat = f(1)^\flat \alpha_0^{x_0}\alpha_1^{x_1}\dotsm\alpha_n^{x_n}
\end{align*}
with $\alpha_k=c_k \prod_{i=0}^{k-1}c_i^{p^{k-1-i}(p-1)}$.
Since $\lim_{n\ra\infty}f(a_n+1)$ converges, so do $\lim_{n\ra\infty}f(a_n+1)^\flat$ and $\prod_{k=0}^\infty \alpha_k^{x_k}$. 
Moreover we can write
\begin{align*}
f(z)\equiv f(1)\prod_{k=0}^\infty \alpha_k^{x_k}.
\end{align*}
Consider the case of $d=2$, $z=\frac{1}{2}$ of (\ref{fe}): $f(\frac{1}{2})f(1) \equiv f(1)$.
Therefore, noting that $-\frac{1}{2}=\sum_{k=0}^\infty \frac{p-1}{2}p^k$, we obtain
\begin{align*}
1 \equiv f(\tfrac{1}{2}) \equiv f(1)\prod_{k=0}^\infty \alpha_k^{\frac{p-1}{2}}, \quad \text{that is,} \quad 
f(1) \equiv \prod_{k=0}^\infty \alpha_k^{-\frac{p-1}{2}}.
\end{align*}  
Then the assertion (ii) is also clear.

Next, assume (\ref{explicit}).
When $\op z=k$, we see that $\frac{f(z+1)}{f(z)}\equiv \frac{\alpha_k}{\alpha_{k-1}^{p-1}}$ (resp.\ $\alpha_0$) if $k>0$ (resp.\ $k=0$).
In particular, $g(z):=\frac{f(z+1)}{f(z)} \bmod \mu_\infty$ depends only on $\op z$.
When $z+z'=1$, the $p$-adic expansions $z-1=\sum_{k=0}^\infty x_k p^k$, $z'-1=\sum_{k=0}^\infty x_k' p^k$ satisfy $x_k+x_k'=p-1$ for any $k$.
Then we have
\begin{align*}
f(z)f(z')\equiv \prod_{k=0}^\infty \alpha_k^0= 1.
\end{align*}
Therefore the case $z=0$ of (\ref{fe}) holds true since $\left(\prod_{k=1}^{d-1}f(\frac{k}{d})\right)^2=\prod_{k=1}^{d-1}f(\frac{k}{d})f(1-\frac{k}{d}) \equiv 1$.
Then (\ref{fe}) for $z \in \mathbb N$ follows by mathematical induction on $z$ noting that 
\begin{align*}
\prod_{k=0}^{d-1}f(z+1+\tfrac{k}{d})&\equiv \prod_{k=0}^{d-1}f(z+\tfrac{k}{d})g(z+\tfrac{k}{d}), \\
f(dz+d) &\equiv  f(dz)g(dz)\dotsm g(dz+d-1), \\
\op (dz+k)&= \op (z+\tfrac{k}{d}).
\end{align*}
Since $\mathbb N$ is dense in $\mathbb Z_p$, we see that (\ref{fe}) holds for any $z \in \mathbb Z_p$.
\end{proof}

The following corollary provides a nice characterization of $\Gamma_p(z) \bmod \mu_\infty$ in terms of functional equations and one or two special values.

\begin{crl} \label{crloffe}
Assume a continuous function $f(z) \colon \mathbb Z_p \ra \mathbb C_p^\times$ satisfies 
\begin{align*}
\prod_{k=0}^{d-1} f(z+\tfrac{k}{d})\equiv f(dz) \mod \mu_\infty \quad (p\nmid d)
\end{align*}
and put 
\begin{align*}
c_n:=\left(\frac{f(p^n+1)}{f(p^n)}\right)^\flat.
\end{align*}
Then the following equivalences hold:
\begin{enumerate}
\item $c_0= c_1 =\cdots$ $\LR$ $f(z)\equiv c_0^{z-\frac{1}{2}} \bmod \mu_\infty$. 
\item $c_1=c_2=\cdots$ $\LR$ $f(z)\equiv c_0^{z-\frac{1}{2}}(c_1/c_0)^{z_1+\frac{1}{2}} \bmod \mu_\infty$.
\end{enumerate}
\end{crl} 

\begin{proof} 
We suppress $\bmod \mu_\infty$. 
For (i), assume that $c_0= c_1 = \cdots$.
Then 
\begin{align*}
\alpha_k:= c_k \prod_{i=0}^{k-1}c_i^{p^{k-1-i}(p-1)} =c_0^{p^k}.
\end{align*}
Hence we have by Proposition \ref{feef}
\begin{align*}
f\left(1+\sum_{k=0}^\infty x_kp^k\right) \equiv \prod_{k=0}^\infty \alpha_k^{x_k-\frac{p-1}{2}} =c_0^{\sum_{k=0}^\infty x_kp^k-\frac{p-1}{2}p^k}= c_0^{z-1+\frac{1}{2}} =c_0^{z-\frac{1}{2}}.
\end{align*}
The opposite direction is trivial by definition $c_n:=(\frac{f(p^n+1)}{f(p^n)})^\flat$.
For (ii), the assumption $c_1= c_2 = \cdots$ implies $\alpha_0=c_0$, $\alpha_k =c_0^{p^k}(c_1/c_0)^{p^{k-1}}$ ($k\geq 1$).
In this case we have
\begin{align*}
f\left(1+\sum_{k=0}^\infty x_kp^k\right) \equiv c_0^{\sum_{k=0}^\infty x_kp^k-\frac{p-1}{2}p^k}(c_1/c_0)^{\sum_{k=1}^\infty x_kp^{k-1}-\frac{p-1}{2}p^{k-1}}
=c_0^{z-\frac{1}{2}}(c_1/c_0)^{z_1+\frac{1}{2}}
\end{align*}
since $\sum_{k=1}^\infty x_kp^{k-1}=\frac{z-1-x_0}{p}=z_1$.
\end{proof}

\subsection{Alternative proof of a part of Coleman's formula}

We fix $\tau\in \wg$ with $\deg \tau=1$ and put 
\begin{align} 
G_1(z)&:=\left(p^{\frac{1}{2}-\tau^{-1}(z)}\frac{P(z)}{\Phi_{\tau}(P(\tau^{-1}(z)))}\right)^\flat && (z \in \mathbb Z_{(p)} \cap (0,1)), \label{Gz} \\ 
G_2(z)&:=\left(\frac{p^{(\tau^{-1}(z)-z)\op z}P(z)}{\Phi_\tau (P(\tau^{-1}(z)))}\right)^\flat && (z \in (\mathbb Q-\mathbb Z_{(p)}) \cap (0,1)). \notag 
\end{align}
Here we added $(\ )^\flat$ to the right-hand sides of Coleman's formulas (Theorem \ref{Cf}), in order to resolve a root of unity ambiguity, only superficially.
Note that $G_2$ corresponds to Theorem \ref{Cf}-(ii) replaced $z$ with $\tau^{-1}(z)$.

By Theorem \ref{Cf}-(i), we see that $G_1$ is continuous for the $p$-adic topology.
$G_2$ is not $p$-adically continuous in the usual sense, on the whole of $(\mathbb Q-\mathbb Z_{(p)}) \cap (0,1)$ (for details, see Remark \ref{notcnt}).
Theorem \ref{Cf}-(i) only implies the following ``continuity'':
\begin{align} 
&\text{$G_1(z)$ is continuous for the relative topology} \notag \\
&\text{induced by $z \in (\mathbb Q-\mathbb Z_{(p)}) \cap (0,1) \hookrightarrow \mathbb Q_p \times\mathbb Q_p$, $z \mt (z,\tau^{-1}(z))$.} \label{cntforG2}
\end{align}

In Corollary \ref{maincrl}, oppositely, we show that the $p$-adic continuity of $G_1,G_2$ implies a ``large part'' 
\begin{align*}
G_1(z) \equiv  a^{z-\frac{1}{2}}  b^{z_1+\frac{1}{2}} \Gamma_p(z) \mod \mu_\infty \quad (a,b \in \mathbb C_p^\times)
\end{align*}
of Theorem \ref{Cf}-(i):
\begin{align*}
G_1(z)\equiv \Gamma_p(z) \mod \mu_\infty.
\end{align*}
Besides we shall show the continuity of $G_1(z)$ in \S \ref{secpc}, independently of Theorem \ref{Cf}.

Hereinafter in this section, we forget Theorem \ref{Cf}.
We assume the following Assumption instead.
\begin{asm} \label{asm}
$G_1(z)$ is $p$-adically continuous and $G_2(z)$ is continuous in the sense of {\rm(\ref{cntforG2})}.
In particular, we regard $G_1$ as a $p$-adic continuous function:
\begin{align*} 
G_1(z) \colon \mathbb Z_p \ra \mathbb C_p.
\end{align*}
\end{asm}

First we derive ``multiplication formula'':
\begin{align} \label{mf4G}
\prod_{k=0}^{d-1} G_1(z+\tfrac{k}{d}) \equiv d^{1-dz+(dz)_1} G_1(dz) \mod \mu_\infty \quad (p \nmid d \in \mathbb N)
\end{align} 
independently of Theorem \ref{Cf}.

\begin{proof}[Proof of {\rm(\ref{mf4G})}]  
We suppress $\bmod \mu_\infty$. 
Let $z\in \mathbb Z_{(p)} \cap (0,\frac{1}{d})$.
By Definition \ref{dfn}-(viii) and (\ref{Gz}) we can write
\begin{align*}
\frac{\prod_{k=0}^{d-1}G_1(z+\frac{k}{d})}{G_1(dz)}
&\equiv\frac{\prod_{k=0}^{d-1}\Gamma_\infty(z+\frac{k}{d})}{\Gamma_\infty(dz)}
\Phi_\tau\left(\frac{\Gamma_\infty(\tau^{-1}(dz))}{\prod_{k=0}^{d-1}\Gamma_\infty(\tau^{-1}(z+\frac{k}{d}))} \right)
\frac{\prod_{k=0}^{d-1}p^{\frac{1}{2}-\tau^{-1}(z+\frac{k}{d})}}{p^{\frac{1}{2}-\tau^{-1}(dz)}}  \\
&\quad \times \text{``products of classical or $p$-adic periods''},
\end{align*}
where the ``products of classical or $p$-adic periods'' become trivial by (\ref{mffb1}), as we saw in the proof of Proposition \ref{toy}.
Besides we see that 
\begin{align*}
\{\tau^{-1}(z+\tfrac{k}{d})\mid k=0,\dotsm,d-1\}=\{\tfrac{\tau^{-1}(dz)}{d}+\tfrac{k}{d} \mid k=0,\dotsm,d-1\}.
\end{align*}
To see this, it suffices to show that
$\{\tau^{-1}(\zeta_N^a\zeta_d^k) \mid k=0,\dots,d-1\}$ and $\{\tau^{-1}(\zeta_N^{da})^{\frac{1}{d}}\zeta_d^k \mid k=0,\dots,d-1\}$ coincide with each other.
We easily see that both of them are the inverse image of $\tau^{-1}(\zeta_N^{da})$ under the $d$th power map $\mu_\infty \ra \mu_\infty$, $x\mt x^d$.
Hence we obtain 
\begin{align*}
\frac{\prod_{k=0}^{d-1}G_1(z+\frac{k}{d})}{G_1(dz)}
&\equiv\frac{\prod_{k=0}^{d-1}\Gamma_\infty(z+\frac{k}{d})}{\Gamma_\infty(dz)} \cdot 
\Phi_\tau\left(\frac{\Gamma_\infty(\tau^{-1}(dz))}{\prod_{k=0}^{d-1}\Gamma_\infty(\frac{\tau^{-1}(dz)}{d}+\tfrac{k}{d})} \right)
\cdot \frac{\prod_{k=0}^{d-1}p^{\frac{1}{2}-(\tfrac{\tau^{-1}(dz)}{d}+\tfrac{k}{d})}}{p^{\frac{1}{2}-\tau^{-1}(dz)}} \\
&= d^{\frac{1}{2}-dz} \cdot \Phi_\tau( d^{\tau^{-1}(dz)-\frac{1}{2}}) \cdot 1 \equiv d^{\frac{1}{2}-dz} \cdot d^{\tau^{-1}(dz)-\frac{1}{2}}
\end{align*}
by (\ref{gmf}), (\ref{mffb1}).
For the last ``$\equiv$'', we note that $\Phi_\tau$ acts on $\overline{\mathbb Q_p} \ni d^{\tau^{-1}(dz)-\frac{1}{2}}$ as $\tau$.
By Remark \ref{rmk4dfn}-(ii), we have $\tau^{-1}(dz)=(dz)_1+1$.
Then the assertion is clear.
\end{proof}

Furthermore we can show that $c_n=\left(\frac{f(p^n+1)}{f(p^n)}\right)^\flat$ for $f(z):=\frac{G_1(z)}{\Gamma_p(z)}$ is constant, at least for $n\geq 1$.

\begin{thm} \label{mthm}
We assume {\rm Assumption \ref{asm}} and put $f(z):=\frac{G_1(z)}{\Gamma_p(z)}$.
\begin{enumerate}
\item The following functional equations hold.
\begin{align*}
\prod_{k=0}^{d-1} f(z+\tfrac{k}{d})\equiv f(dz)  \mod \mu_\infty \quad (p\nmid d).
\end{align*}
\item We have $c_1= c_2 = \cdots$ for $c_n:=\left(\tfrac{f(p^n+1)}{f(p^n)}\right)^\flat$.
\end{enumerate}
\end{thm} 

\begin{proof}
We suppress $\bmod \mu_\infty$. 
(i) follows from (\ref{dmult}), (\ref{mf4G}).
For (ii), we need for $z \in p\mathbb Z_p$
\begin{align*}
\frac{G_1(pz)G_1(z+1)}{G_1(pz+1)G_1(z)} \equiv \frac{\Gamma_p(pz)\Gamma_p(z+1)}{\Gamma_p(pz+1)\Gamma_p(z)}.
\end{align*}
Since the right-hand side is equal to $\begin{cases}1 &(p\mid z) \\ z &(p\nmid z) \end{cases}$ by (\ref{cofgp}), it suffices to show that
\begin{align*}
\frac{G_1(pz)G_1(z+1)}{G_1(pz+1)G_1(z)} \equiv 1 \quad (z \in p\mathbb Z_p).
\end{align*}
Note that we can not use the definition (\ref{Gz}) directly since $z,z+1,pz,pz+1$ are not contained in $(0,1)$ simultaneously.
Therefore a little complicated argument is needed as follows.
Let $z \in \mathbb Z_{(p)} \cap(0,\frac{1}{p})$. 
By Remark \ref{rmk4dfn}-(ii), we have
\begin{align*}
\tau(z)=\langle p z \rangle=pz, \text{ hence } \tau^{-1}(pz)=z.
\end{align*}
We can write
\begin{align*}
H_1(z)&:=\frac{G_1(z)G_2(z+\frac{1}{p})\dotsm G_2(z+\frac{p-1}{p})}{G_1(pz)} \\
&\equiv p^{z+(z+\frac{1}{p})+\dots+(z+\frac{p-1}{p})-\tau^{-1}(z)-\tau^{-1}(z+\frac{1}{p})-\dots-\tau^{-1}(z+\frac{p-1}{p})} \frac{P(z)P(z+\frac{1}{p})\dotsm P(z+\frac{p-1}{p})}{P(pz)}\\
& \quad \times \Phi_{\tau}\left(\frac{P(z)}{P(\tau^{-1}(z))P(\tau^{-1}(z+\frac{1}{p}))\dotsm P(\tau^{-1}(z+\frac{p-1}{p}))}\right).
\end{align*}
Here we note that $\mathrm{ord}_p(z+\frac{k}{p})=-1$ for $k=1,\dots,p-1$.
We have 
\begin{align} \label{tinv}
\{\tau^{-1}(z+\tfrac{k}{p}) \mid k=0,\dots,p-1\}=\{\tfrac{z+k}{p} \mid k=0,\dots,p-1\}
\end{align}
since both of $\{\tau^{-1}(\zeta_N^a \zeta_p^k) \mid k=0,\dots,p-1\}$, $\{\zeta_{pN}^{a+Nk} \mid  k=0,\dots,p-1\}$ are the  set of the $p$th roots of $\zeta_N^a$ when $z=\frac{a}{N}$.
Therefore the $p$-power parts of $H_1$ become 
\begin{align*}
p^{z+(z+\frac{1}{p})+\dots+(z+\frac{p-1}{p})-\frac{z}{p}-\frac{z+1}{p}-\dots-\frac{z+p-1}{p}}=p^{(p-1)z}.
\end{align*}
Moreover the ``period parts'' of $H_1$ become trivial by (\ref{mffb1}), (\ref{tinv}).
Namely we can write 
\begin{align*}
H_1(z)\equiv  p^{(p-1)z}\frac{\Gamma_\infty(z)   \Gamma_\infty(z+\frac{1}{p})\dotsm \Gamma_\infty(z+\frac{p-1}{p})}{\Gamma_\infty(pz)}
\Phi_{\tau}\left(\frac{\Gamma_\infty(z)}{\Gamma_\infty(\frac{z}{p})\Gamma_\infty(\frac{z+1}{p})\dotsm \Gamma_\infty(\frac{z+p-1}{p})}\right).
\end{align*}
By using the original Multiplication formula (\ref{gmf}) for $\Gamma_\infty$, we obtain
\begin{align*}
H_1(z)\equiv  p^{(p-1)z} p^{\frac{1}{2}-pz} p^{z-\frac{1}{2}}=1.
\end{align*}

Next, let $z=\frac{a}{N} \in \mathbb Z_{(p)}\cap (-\frac{1}{p},0)$.
Then we have
\begin{itemize}
\item $\tau(z+1)=pz+1$. Hence $\tau^{-1}(pz+1)=z+1$.
\item $\{\tau^{-1}(\zeta_N^a \zeta_p^k) \mid k=1,\dots,p\}=\{\zeta \mid \zeta^p=\zeta_N^a\}=\{\zeta_{pN}^{a+Nk} \mid  k=1,\dots,p\}$.
Hence $\{\tau^{-1}(z+\frac{k}{p})\mid k=1,\dots,p\}=\{\tfrac{z+k}{p}\mid k=1,\dots,p\}$.
\end{itemize}
Then we can prove similarly that 
\begin{align*}
&H_2(z)\\
&:=\frac{G_2(z+\frac{1}{p})\dotsm G_2(z+\frac{p-1}{p})G_1(z+1)}{G_1(pz+1)} \\
&\equiv p^{(z+\frac{1}{p})+\dots+(z+\frac{p-1}{p})+(z+1)-\tau^{-1}(z+\frac{1}{p})-\dots-\tau^{-1}(z+\frac{p-1}{p})-\tau^{-1}(z+1)} \frac{P(z+\frac{1}{p})\dotsm P(z+\frac{p-1}{p})P(z+1)}{P(pz+1)} \\
& \quad \times \Phi_{\tau}\left(\frac{P(z+1)}{P(\tau^{-1}(z+\frac{1}{p}))\dotsm P(\tau^{-1}(z+\frac{p-1}{p}))P(\tau^{-1}(z+1))}\right) \\
&\equiv p^{(z+\frac{1}{p})+\dots+(z+\frac{p-1}{p})+(z+1)-\frac{z+1}{p}-\dots -\frac{z+p-1}{p}-\frac{z+p}{p}} p^{\frac{1}{2}-(pz+1)} p^{z+1-\frac{1}{2}} = 1.
\end{align*}

Here $H_i(z)\equiv 1 \bmod \mu_\infty$ implies $H_i(z)= 1$ ($i=1,2$) since 
we have $x^\flat=\exp_p(\log_p x)=\exp_p(0)=1$ for $x\in \mu_\infty$.
($G_1(z),G_2(z)$ are in the image under $(\ )^\flat$ by definition, so are $H_i(z)$.)
In particular, we have
\begin{align*}
\frac{G_1(pz)}{G_1(z)}&= G_2(z+\tfrac{1}{p})\dotsm G_2(z+\tfrac{p-1}{p}) \quad (z \in \mathbb Z_{(p)} \cap(0,\tfrac{1}{p})), \\
\frac{G_1(pz+1)}{G_1(z+1)}&= G_2(z+\tfrac{1}{p})\dotsm G_2(z+\tfrac{p-1}{p}) \quad (z \in \mathbb Z_{(p)}\cap (-\tfrac{1}{p},0)).
\end{align*}

Let $z\in p\mathbb Z_{(p)}$.
Then there exist $z_n^+\in p\mathbb Z_{(p)} \cap (0,\frac{1}{p})$, $z_n^-\in p\mathbb Z_{(p)} \cap (-\frac{1}{p},0)$ which converge to $z$ when $n\ra \infty$ respectively.
Then we can write
\begin{align*}
\frac{G_1(pz)}{G_1(z)}
&=\lim_{n\ra \infty} \frac{G_1(pz_n^+)}{G_1(z_n^+)}
= \lim_{n\ra \infty} G_2(z_n^+ +\tfrac{1}{p})\dotsm G_2(z_n^+ +\tfrac{p-1}{p}), \\
\frac{G_1(pz+1)}{G_1(z+1)}&=\lim_{n\ra \infty} \frac{G_1(pz_n^- +1)}{G_1(z_n^- +1)}
= \lim_{n\ra \infty} G_2(z_n^- +\tfrac{1}{p})\dotsm G_2(z_n^- +\tfrac{p-1}{p}).
\end{align*}
Recall that $G_2(z)$ is continuous in the sense of (\ref{cntforG2}).
Clearly we have for $k=1,\dots,p-1$
\begin{align*}
z_n^{\pm}+\tfrac{k}{p} \ra z+ \tfrac{k}{p} \quad (n\ra \infty).
\end{align*}
Additionally we see that 
\begin{align*}
\tau^{-1}(z_n^{\pm}+\tfrac{k}{p})=
\tfrac{z_n^{\pm}}{p}+\tau^{-1}(\tfrac{k}{p}) \ra \tfrac{z}{p}+\tau^{-1}(\tfrac{k}{p})
\quad (n\ra \infty)
\end{align*}
by noting that $\tau^{-1}(z+z')\equiv \tau^{-1}(z)+\tau^{-1}(z') \bmod \mathbb Z$ ($\forall z,z'$), $\tau^{-1}(z) \equiv \tfrac{z}{p} \bmod \mathbb Z$ if $p\mid z$, 
$\tfrac{z_n^{\pm}}{p} \in (-\frac{1}{p},\frac{1}{p})$, $\tau^{-1}(\tfrac{k}{p}) \in [\frac{1}{p},\frac{p-1}{p}]$.
It follows that 
\begin{align*}
\lim_{n\ra \infty} G_2(z_n^+ +\tfrac{k}{p})= \lim_{n\ra \infty} G_2(z_n^- +\tfrac{k}{p}).
\end{align*}
Then the assertion is clear.
\end{proof}

By Corollary \ref{crloffe}, we obtain the following.

\begin{crl} \label{maincrl}
Assume {\rm Assumption \ref{asm}}.
Then there exist constants $a,b$ satisfying
\begin{align*}
G_1(z) \equiv  a^{z-\frac{1}{2}}  b^{z_1+\frac{1}{2}} \Gamma_p(z) \mod \mu_\infty.
\end{align*}
\end{crl}  

\begin{rmk} \label{onecurve}
In addition to the above results, by computing the absolute Frobenius on only one Fermat curve, 
we obtain Coleman's formula $G_1(z) \equiv \Gamma_p(z) \bmod \mu_\infty$. 
For example, when $p=3$, we obtain it for $z=\frac{1}{5},\frac{2}{5}$ by the computation on $F_5$.
It follows that $a^{\frac{-3}{10}}b^{\frac{-1}{10}}\equiv a^{\frac{-1}{10}}b^{\frac{3}{10}} \equiv 1$, hence $a\equiv b \equiv 1$. 
\end{rmk}

\begin{rmk} \label{notcnt}
We used the assumption $p \mid z$ only in the last paragraph of the proof for {\rm Theorem \ref{mthm}}
because $G_2$ is not $p$-adically continuous on the whole of $(\mathbb Q-\mathbb Z_{(p)}) \cap (0,1)$.
For example, we put
\begin{align*}
z_n:=\frac{1}{p^2}+\frac{p^{n+1}}{p^{n+2}+(1-p)^n} \in (\mathbb Q-\mathbb Z_{(p)}) \cap (0,1) \quad (n \in \mathbb N)
\end{align*}
and take $\tau \in W_p$ with $\deg \tau =1$ so that
\begin{align*}
\tau(\zeta_{p^2})=\zeta_{p^2}^{-1}.
\end{align*}
In particular we see that 
\begin{center}
$z_n \ra \frac{1}{p^2}$ for the $p$-adic topology. 
\end{center}
On the other hand we see that
\begin{align*}
&\tau^{-1}(z_n) \equiv \tau^{-1}(\tfrac{1}{p^2})+\tau^{-1}(\tfrac{p^{n+1}}{p^{n+2}+(1-p)^n} )=\tfrac{p^2-1}{p^2}+\tfrac{p^{n}}{p^{n+2}+(1-p)^n}=1-\tfrac{(1-p)^n}{p^2(p^{n+2}+(1-p)^n)} 
\bmod \mathbb Z, \\
&1-\tfrac{(1-p)^n}{p^2(p^{n+2}+(1-p)^n)} \in 
\begin{cases}
(1,2) & \text{if $n$ is odd}, \\
(0,1) & \text{if $n$ is even}.
\end{cases}
\end{align*}
Hence we have
\begin{align*}
\tau^{-1}(z_n) =
\begin{cases}
-\tfrac{(1-p)^n}{p^2(p^{n+2}+(1-p)^n)} \ra -\tfrac{1}{p^2} & \text{if $n=2k+1$, $k\ra \infty$}, \\
1-\tfrac{(1-p)^n}{p^2(p^{n+2}+(1-p)^n)} \ra 1-\tfrac{1}{p^2} & \text{if $n=2k$, $k\ra \infty$}.
\end{cases}
\end{align*}
Then, by {\rm Theorem \ref{Cf}-(ii)}, we see that $G_2(z_n) = (\Gamma_p(z_n)/\Gamma_p(\tau^{-1}(z_n)))^\flat $ does not converge $p$-adically although $z_n$ does.
\end{rmk}  

\section{On the $p$-adic continuity} \label{secpc}

In the previous section, we showed that the $p$-adic continuity of the right-hand sides of Theorem \ref{Cf}-(i), (ii) implies a large part of Theorem \ref{Cf}-(i) itself.
In this section, we see that it is relatively easy to show such $p$-adic continuity properties, without explicit computation.
For simplicity, we consider only the case $z \in \mathbb Z_p$.
Assume that $p\nmid N$.
 
\begin{lmm}[{\cite[\S VI]{Co1}}] 
Let $1\leq r,s<N$ with $r+s \neq N$.
We consider the formal expansion of the differential form $\eta_{r,s}=x^ry^{s-N} \frac{dx}{x}$ on $F_N\colon x^N+y^N=1$ at $(x,y)=(0,1)$:
\begin{align*}
\eta_{r,s}&=\sum_{n=0}^\infty b_{r,s}(n) x^n \frac{dx}{x}, \\
b_{r,s}(n)&:=
\begin{cases}
(-1)^{\frac{n-r}{N}}\dbinom{\frac{s}{N}-1}{\frac{n-r}{N}} & (n\equiv r \mod N), \\
0 & (n\not\equiv r \mod N).
\end{cases}
\end{align*}
Let $\Phi$ be the absolute Frobenius on $H_{\mathrm dR}^1(F_N,\mathbb Q_p)$. Then there exists $\alpha_{r',s'} \in \mathbb Q_p$ satisfying  
\begin{align*}
\Phi(\eta_{r,s})=\alpha_{r',s'} \eta_{r',s'} \ \text{for} \ r',s' \ \text{with} \ 1\leq r',s'<N,\ pr\equiv r' \bmod N,\ ps\equiv s' \bmod N.
\end{align*}
Then we have
\begin{align} \label{limit}
\alpha_{r',s'}
=\lim_{\substack{\mathbb N \ni n \mt 0 \\ n\equiv r \bmod N}}\frac{pb_{r,s}(n)}{b_{r',s'}(pn)}
=\lim_{\mathbb N \ni k \ra -\frac{r}{N}} (-1)^{(p-1)k+\frac{pr-r'}{N}}\frac{p\dbinom{\frac{s}{N}-1}{k}}{\dbinom{\frac{s'}{N}-1}{pk+\frac{pr-r'}{N}}}.
\end{align}
\end{lmm}
We note that $\alpha_{r',s'}$ depends only on $(\frac{r'}{N},\frac{s'}{N})$.
That is $\alpha_{r',s'}$ with $N=N_1$ is equal to $\alpha_{tr',ts'}$ with $N=tN_1$.

\begin{prp} \label{Friscont}
$\alpha_{r',s'}$ is $p$-adically continuous on $(\frac{r'}{N},\frac{s'}{N}) \in (\mathbb Z_{(p)} \cap (0,1))^2$.
\end{prp}

\begin{proof}
It suffices to show that $\alpha_{r'_1,s'_1}$ with $N=N_1$ is close to $\alpha_{r'_2,s'_2}$ with $N=N_2$ 
when $\frac{r'_1}{N_1}$ is close to $\frac{r'_2}{N_2}$ and $\frac{s'_1}{N_1}$ is close to $\frac{s'_2}{N_2}$.
We may assume $N:=N_1=N_2$ by considering $N=N_1N_2$. 
First we fix $r':=r'_1=r'_2$ and assume that $s_1'$ is close to $s_2'$.
Then we can take the same $k$ for the limit expressions (\ref{limit}) of $\alpha_{r',s'_1}$, $\alpha_{r',s'_2}$.
We easily see that if $p^l \mid (s_1'-s_2')$, then $p^{l-1} \mid (s_1-s_2)$. In fact, we can write $s_i'=ps_i-l_iN$ with $l_i=0,1,\dots,p-1$ since $0< s_i,s_i' < N$ for $i=1,2$.
If $p \mid (s_1'-s_2')$, then we have $p \mid (l_1-l_2)$, so $l_1=l_2$. Therefore we obtain $s_1-s_2=\frac{s_1'-s_2'}{p}$.
It follows that $s_1$ also is close to $s_2$.
Hence the continuity on $\frac{s'}{N}$ is clear since the numerator (resp.\ the denominator) of the expression (\ref{limit}) is a polynomial on $\frac{s}{N}$ (resp.\ $\frac{s'}{N}$).

For the variable $\frac{r'}{N}$, we replace $x$ with $y$.
In other words, replace the point $(x,y)=(0,1)$ for the expansion with $(1,0)$.
Then the continuity on $\frac{r'}{N}$ also follows from the same argument.
\end{proof}

\begin{crl} \label{Gisc}
$G_1(z)$ defined in {\rm (\ref{Gz})} is $p$-adically continuous on $z \in \mathbb Z_{(p)} \cap (0,1)$.
In particular, we may regard $G_1(z)$ as a continuous function on $\mathbb Z_p$.
\end{crl}

\begin{proof}
CM-types $\Xi_{r,s}$ of (\ref{xi}), corresponding to $\eta_{r,s}$, generate the $\mathbb Q$-vector space 
$\{\sum_{\sigma} c_\sigma \cdot \sigma \mid c_{\sigma}+c_{\rho \circ \sigma}$ is a constant$\}$.
More explicitly, we claim that  
\begin{align*}
\sum_{(b,N)=1} \left(\tfrac{1}{2}-\langle \tfrac{ab}{N}\rangle\right) \sigma_b
=\tfrac{1}{N}\sum_{1\leq s<N,\ a+s\neq N}\Xi_{a,s}-\tfrac{N-2}{2N}\sum_{(b,N)=1} \sigma_b,
\end{align*}
where $s$ runs over $1\leq s<N$ with $a+s\neq N$ in the first sum of the right-hand side. 
By the definition (\ref{xi}), $\sigma_b \in \Xi_{a,s}$ if and only if $\langle \frac{ab}{N}\rangle+\langle \frac{sb}{N}\rangle<1$.
Namely $\langle \frac{sb}{N}\rangle=\frac{1}{N},\frac{2}{N},\dots,1-\frac{1}{N}-\langle\frac{ab}{N}\rangle$.
The number of such $b$ is congruent  to $-1-ab \bmod N$. 
Hence we have
\begin{align*}
\tfrac{1}{N}\sum_{1\leq s<N,\ a+s\neq N}\Xi_{a,s}= \sum_{(b,N)=1} \langle \tfrac{-1-ab}{N}\rangle \sigma_b=\sum_{(b,N)=1} \left(1-\tfrac{1}{N} -\langle \tfrac{ab}{N}\rangle\right) \sigma_b.
\end{align*}
Here we note that $ab \not \equiv 0 \bmod N$ since $(b,N)=1$, $a\not \equiv 0 \bmod N$.
Then the above claim follows. By substituting this into Definition \ref{dfn}-(viii), we can write
\begin{align*}
&P(\tfrac{a}{N}) \equiv \frac{\Gamma_\infty(\frac{a}{N})(2\pi i)_p^{\frac{1}{2}-\frac{a}{N}}\prod_{1\leq s<N,\ a+s\neq N} \left((2\pi i)_p^{e_s}\int_{\gamma,p} \eta_{a,s}\right)^{\frac{1}{N}}}
{(2\pi i)^{\frac{1}{2}-\frac{a}{N}}\prod_{1\leq s<N,\ a+s\neq N} \left((2\pi i)^{e_s}\int_{\gamma} \eta_{a,s}\right)^{\frac{1}{N}}} \mod \mu_\infty, \\
&e_s:=
\begin{cases}
-1 & (a+s<N) \\
0 & (a+s>N)
\end{cases}
\end{align*}
since the part $\sum_{(b,N)=1} \sigma_b$ becomes trivial by Proposition \ref{tips}-(ii).
We can strengthen the congruence relation $\equiv$ of the formula (\ref{gamma}) into an equality $=$, 
by selecting a specific closed path $\gamma_0$ (e.g., $\gamma_0=N\gamma_N$ with $\gamma_N$ in \cite[Proposition 4.9]{Ot}).
Then we have
\begin{align*}
P(\tfrac{a}{N}) \equiv c \cdot (2\pi i)_p^{\frac{-1}{2}+\frac{1}{N}} \prod_{1\leq s<N,\ a+s\neq N} \left(\int_{\gamma_0,p} \eta_{a,s}\right)^{\frac{1}{N}} \mod \mu_\infty,
\end{align*}
where we put
\begin{align*}
c:=\frac{\Gamma(\frac{a}{N})}{(2\pi)^{\frac{1}{N}}}\left(\prod_{1\leq s < N,\ a+s\neq N} \frac{\Gamma(\frac{a+s}{N})}{\Gamma(\frac{a}{N})\Gamma(\frac{s}{N})}\right)^\frac{1}{N}.
\end{align*}
Since (\ref{gmf}) implies that
\begin{align*}
\prod_{1\leq s \leq N} \frac{\Gamma(\frac{a+s}{N})}{\Gamma(\frac{a}{N})\Gamma(\frac{s}{N})} 
=\frac{N^{-a}a!}{\Gamma(\frac{a}{N})^N},
\end{align*}
we obtain 
\begin{align*}
c=\frac{\Gamma(\frac{a}{N})}{(2\pi)^{\frac{1}{N}}} 
\left(\frac{\Gamma(\frac{a}{N})\Gamma(\frac{N-a}{N})}{\Gamma(1)} \frac{\Gamma(\frac{a}{N})\Gamma(\frac{N}{N})}{\Gamma(\frac{a+N}{N})}\frac{N^{-a}a!}{\Gamma(\frac{a}{N})^N}\right)^\frac{1}{N}
=\left(\frac{N^{1-a}(a-1)!}{2\sin (\frac{a}{N}\pi)}\right)^\frac{1}{N}.
\end{align*}
For the last equality we used (\ref{erf}) and the difference equation $\Gamma(z+1)=z\Gamma(z)$.
Take $\tau \in W_p$ with $\deg \tau =1$. Then we have
\begin{align*}
G_1(\tfrac{a'}{N}) &\equiv p^{\frac{1}{2}-\frac{a}{N}} \frac{P(\frac{a'}{N})}{\Phi_\tau(P(\frac{a}{N}))} 
\equiv \left(\frac{N^{a-a'}(a'-1)!}{ p^{a-N-1} (a-1)!}\prod_{1\leq s<N,\ a+s\neq N} \alpha_{a',s'}^{-1} \right)^\frac{1}{N}\mod \mu_\infty,
\end{align*}
by noting that $\Phi_\tau((2\pi i)_p)=p(2\pi i)_p$ and $\Phi_\tau(\sin(\tfrac{a}{N}\pi))=\tau(\sin(\tfrac{a}{N}\pi))=\pm \sin(\tfrac{a'}{N} \pi)$.
Here $a',s'$ denote integers satisfying $1\leq a',s' <N$, $pa \equiv a' \bmod N$, $ps\equiv s' \bmod N$ as above.
By Proposition \ref{Friscont}, $\alpha_{a',s'}$ are continuous for $a'$.
When $a$ is in a small open ball, as we saw in the proof of  Proposition \ref{Friscont}, we may write $a'=pa-M$ for a fixed $M$ ($M$ is $lN$ in the proof of  Proposition \ref{Friscont}).
Then the remaining part becomes
\begin{align*}
\frac{N^{a-a'}(a'-1)!}{ p^{a-N-1} (a-1)!}=\pm  \Gamma_p(a'+M+1) \frac{p^{N}N^{\frac{(1-p)a'+M}{p}}(a'+M)}{a' (a'+1)(a'+2)\cdots (a'+M)},
\end{align*}
which is also continuous as desired.
\end{proof}

\end{document}